\documentclass[12pt]{amsart}

\usepackage{amsmath,amssymb,amsthm,mathtools}
\usepackage[a4paper,margin=1.05in]{geometry}
\usepackage{enumitem}
\usepackage{hyperref}
\usepackage[nameinlink,capitalize]{cleveref}
\usepackage{microtype}

\hypersetup{colorlinks=true,linkcolor=blue,citecolor=blue,urlcolor=blue}
\newtheorem{theorem}{Theorem}[section]
\newtheorem{proposition}[theorem]{Proposition}
\newtheorem{lemma}[theorem]{Lemma}
\newtheorem{corollary}[theorem]{Corollary}

\newtheorem{remark}[theorem]{Remark}
\theoremstyle{definition}

\crefname{theorem}{Theorem}{Theorems}
\crefname{lemma}{Lemma}{Lemmas}
\crefname{corollary}{Corollary}{Corollaries}
\crefname{proposition}{Proposition}{Propositions}
\crefname{remark}{Remark}{Remarks}
\crefname{claim}{Claim}{Claims}
\Crefname{theorem}{Theorem}{Theorems}
\Crefname{lemma}{Lemma}{Lemmas}
\Crefname{corollary}{Corollary}{Corollaries}
\Crefname{proposition}{Proposition}{Propositions}
\Crefname{remark}{Remark}{Remarks}
\Crefname{claim}{Claim}{Claims}

\newcommand{\C}{\mathbb C}
\newcommand{\D}{\mathbb D}
\newcommand{\Dstar}{\mathbb D^{*}}
\newcommand{\Q}{\mathbb Q}
\newcommand{\R}{\mathbb R}
\newcommand{\Z}{\mathbb Z}
\newcommand{\eps}{\varepsilon}
\newcommand{\dA}{\,dA}
\newcommand{\dV}{\,dV}

\title[$L^p$-boundedness of Forelli--Rudin type operators]
{$L^p$-Boundedness of Forelli--Rudin Type Operators on Rational Hartogs Triangles}

\author{Qian Fu}
\address{Teaching Experiment Platform, Beijing Normal University, Zhuhai 519087, P. R. China}
\email{qianfu@bnu.edu.cn}


\begin{document}

\begin{abstract}
Let
\[
   H_{m/n}=\{(z_1,z_2)\in\C^2:|z_1|^m<|z_2|^n<1\},
   \qquad \gcd(m,n)=1,
\]
be a rational Hartogs triangle. We characterize the $L^p$-boundedness of
Forelli--Rudin type operators associated with its Bergman kernel. For the
operators with kernel $|B_{m/n}(z,w)|^{c/2}$, the characterization holds for
all $a,b\in\R$ and $c>0$; for the operators with kernel $B_{m/n}(z,w)^N$, it
holds for every $N\in\Z_+$. The conditions are necessary and sufficient and
recover the sharp $L^p$-ranges of the Bergman projection and the Berezin
transform.
\end{abstract}

\maketitle

\section{Introduction}

For $\gamma>0$, the generalized Hartogs triangle is
\[
   H_\gamma=\{(z_1,z_2)\in\C^2:|z_1|^\gamma<|z_2|<1\}.
\]
If $\gamma=m/n\in\Q_+$ is written in lowest terms, then
\[
   H_{m/n}=\{(z_1,z_2)\in\C^2:|z_1|^m<|z_2|^n<1\},
   \qquad m,n\in\Z_+.
\]
Throughout, $\Z_+$ denotes the set of positive integers. These domains are basic models for the study of the Bergman projection on
pseudoconvex domains with non-Lipschitz boundary. The interaction between the
origin and the boundary component $|z_1|^m=|z_2|^n$ produces phenomena that do
not occur on smoothly bounded strongly pseudoconvex domains.

The $L^p$-regularity of the Bergman projection on Hartogs triangles has been
studied extensively. Chakrabarti and Zeytuncu obtained the sharp range
$4/3<p<4$ on the classical triangle. Edholm computed Bergman kernels for
several generalized triangles, and Edholm and McNeal treated fat triangles
and developed the arithmetic subspace decomposition for rational exponents
\cite{ChakrabartiZeytuncu2016,Edholm2016,EdholmMcNeal2016,EdholmMcNeal2017}.
In particular, for relatively prime $m,n\in\Z_+$, the Bergman projection on
$H_{m/n}$ is bounded on $L^p$ precisely when
\[
   \frac{2m+2n}{m+n+1}<p<\frac{2m+2n}{m+n-1},
\]
whereas for irrational $\gamma$ it is bounded only on $L^2$
\cite[Theorems~1.1 and~1.2]{EdholmMcNeal2017}. Endpoint, weighted, and
Bergman--Toeplitz estimates on Hartogs triangles were subsequently obtained
in \cite{HuoWickBLMS2020,HuoWickJFA2020,KhanhLiuThuc2019}.

Forelli--Rudin estimates originated in the analysis of integral operators on
the unit ball \cite{ForelliRudin1974,Rudin1980}. They underlie the boundedness
of the Bergman projection, its positive counterpart, the Berezin transform,
and related weighted operators. The corresponding operator theory on the
unit ball was developed further in
\cite{KuresZhu2006,Liu2015,Zhao2015,ZhaoZhou2022,Zhu2007}. On the classical
Hartogs triangle, weighted $L^p$ estimates and a complete $L^p$--$L^q$
analysis were obtained in \cite{QinWangGuo2023,QinWangXie2025}; see also
\cite{GogusSahutoglu2023} for a general criterion for $L^p$-regularity of the
Berezin transform. Related Forelli--Rudin type operators on two distinct
higher-dimensional generalized Hartogs triangles were studied in
\cite{ZouCAOT2026}.

The present paper gives an exact $L^p$ characterization on every rational
triangle $H_{m/n}$ for the operators in \eqref{eq:def-positive} and
\eqref{eq:def-integer} below. For $m>1$, the Bergman kernel is a finite sum of
$m$ arithmetic subkernels rather than a single product kernel. Consequently,
an estimate for one summand does not control the full kernel, and lower
bounds must account for possible cancellation among the summands. We derive
a global majorant and a two-sided diagonal estimate directly from the complete
Edholm--McNeal decomposition. The unique congruence index $j_*$ satisfying
\[
   n(j_*+1)\equiv1\pmod m
\]
determines the exponent
\[
   \kappa=\frac{m+n-1}{m}
\]
that governs the singularity at $(0,0)$. These estimates yield a single Schur
argument for all real weights $a,b$, including negative values.

Lower estimates require a separate localization because the full kernel may
vanish. Edholm and Mathew proved the existence of zeros for broad
rational subfamilies, including infinitely many rational exponents converging
to $1$ \cite[Theorem~1.7 and Corollary~1.8]{EdholmMathew2026}. We localize the
coordinate products in the convex cone $\Gamma_\theta$ defined in
\eqref{eq:cone-definition}. On this region, all nonzero terms in the common
numerator have controlled arguments, so their sum cannot cancel. The same
argument control gives the lower bound for $\operatorname{Re}(B_{m/n}^N)$
needed for the operators with integer powers.

Section~2 establishes the kernel estimates and auxiliary integral lemmas.
Section~3 proves sufficiency by Schur's test. Section~4 proves the necessary
conditions, and Section~5 treats integer powers.

For $a,b\in\R$ and $c>0$, define
\begin{equation}\label{eq:def-positive}
   S^+_{a,b,c}f(z)
   =\int_{H_{m/n}}
   \frac{|B_{m/n}(z,w)|^{c/2}}
        {B_{m/n}(z,z)^{a/2}B_{m/n}(w,w)^{b/2}}
   f(w)\dV(w).
\end{equation}
For $N\in\Z_+$, define
\begin{equation}\label{eq:def-integer}
   S_{a,b,2N}f(z)
   =\int_{H_{m/n}}
   \frac{B_{m/n}(z,w)^N}
        {B_{m/n}(z,z)^{a/2}B_{m/n}(w,w)^{b/2}}
   f(w)\dV(w).
\end{equation}

For $1<p<\infty$, let $\mathcal C(a,b,c;p)$ denote the following three
conditions:
\begin{equation*}\tag{C1}\label{cond:boundary}
   -ap<1<(b+1)p,
\end{equation*}
\begin{equation*}\tag{C2}\label{cond:vertex}
   (c-2a)(m+n-1)p<4(m+n)
   <\big((2b-c)(m+n-1)+4(m+n)\big)p.
\end{equation*}
\begin{equation*}\tag{C3}\label{cond:pinch}
   c\le a+b+2.
\end{equation*}

\begin{theorem}\label{thm:positive-main}
Let $m,n\in\Z_+$ be relatively prime, $1<p<\infty$, $a,b\in\R$, and $c>0$.
Then $S^+_{a,b,c}$ is bounded on $L^p(H_{m/n})$ if and only if
$\mathcal C(a,b,c;p)$ holds.
\end{theorem}

\begin{theorem}\label{thm:integer-main}
Let $m,n\in\Z_+$ be relatively prime, $1<p<\infty$, $a,b\in\R$, and $N\in\Z_+$.
The following are equivalent:
\begin{enumerate}[label=\textup{(\roman*)}]
   \item $S^+_{a,b,2N}$ is bounded on $L^p(H_{m/n})$;
   \item $S_{a,b,2N}$ is bounded on $L^p(H_{m/n})$;
   \item $\mathcal C(a,b,2N;p)$ holds.
\end{enumerate}
\end{theorem}

\begin{corollary}\cite{EdholmMcNeal2017}\label{cor:bergman-projection}
The Bergman projection on $H_{m/n}$ is bounded on $L^p(H_{m/n})$ if and only if
\[
   \frac{2m+2n}{m+n+1}<p<\frac{2m+2n}{m+n-1}.
\]
\end{corollary}

\begin{corollary}\label{cor:berezin}
Let
\[
   \mathcal B_{m/n}f(z)=\int_{H_{m/n}}
   \frac{|B_{m/n}(z,w)|^2}{B_{m/n}(z,z)}f(w)\dV(w).
\]
Then $\mathcal B_{m/n}$ is bounded on $L^p(H_{m/n})$ if and only if
\[
   p>m+n.
\]
\end{corollary}

\begin{remark}
For
$(m,n)=(1,1)$, the criteria reduce to the classical Hartogs-triangle formulas:
Corollary~\ref{cor:bergman-projection} gives $4/3<p<4$, while
Corollary~\ref{cor:berezin} gives $p>2$.  Thus the main theorems agree with
\cite{ChakrabartiZeytuncu2016,QinWangGuo2023}.
\end{remark}

\begin{remark}
The condition $c>0$ in Theorem~\ref{thm:positive-main} is deliberate.  Since
$B_{m/n}$ can have zeros on $H_{m/n}\times H_{m/n}$, negative powers of
$|B_{m/n}|$ would create interior singularities not governed by the boundary
conditions \eqref{cond:boundary}--\eqref{cond:pinch}.  For $S_{a,b,2N}$, the exponent is an integer and only the ordinary power
$B_{m/n}^N$ occurs.
\end{remark}

\section{Kernel estimates and auxiliary lemmas}

If $A$ and $B$ are nonnegative quantities depending on several variables,
$A\lesssim B$ will signify that there exists a constant $C>0$, independent of
the relevant variables, such that $A\le C B$.  We write $A\approx B$ if both
$A\lesssim B$ and $B\lesssim A$ hold. If $x\in\R$, then $\lfloor x\rfloor$
denotes the greatest integer not exceeding $x$.

The symbol $\dV$ denotes Lebesgue measure on $\C^2$, and
$\dA$ denotes planar Lebesgue measure on $\C$.  Their normalizations play
no role in the estimates.

For $\xi\in H_{m/n}$, set
\begin{equation}\label{eq:PhiRrho}
   \Phi(\xi)=\frac{\xi_1^m}{\xi_2^n},\qquad
   R(\xi)=1-|\Phi(\xi)|^2,\qquad
   \rho(\xi)=1-|\xi_2|^2,
\end{equation}
and put
\[
   \kappa=\frac{m+n-1}{m}.
\]
For $0<\theta<\pi/2$, define
\begin{equation}\label{eq:cone-definition}
   \Gamma_\theta
   =\{0\}\cup\{\zeta\in\C\setminus\{0\}:
      |\operatorname{Arg}\zeta|<\theta\},
\end{equation}
where $\operatorname{Arg}\zeta\in(-\pi,\pi]$ is the principal argument.
Thus $\Gamma_\theta$ is the convex cone with vertex at the origin, axis the
positive real line, and aperture $2\theta$.  We shall repeatedly use
\[
   \Gamma_{\theta_1}\Gamma_{\theta_2}
   \subset \Gamma_{\theta_1+\theta_2}
\]
whenever $\theta_1+\theta_2<\pi/2$.

The following decomposition is due to Edholm and McNeal
\cite[Theorem~3.4 and Corollary~3.5]{EdholmMcNeal2017}.
\begin{lemma}\cite{EdholmMcNeal2017}\label{lem:subkernel-formula}
For $0\le j\le m-1$, set
\[
   E_j=\left\lfloor \frac{(j+1)n-1}{m}\right\rfloor.
\]
The Bergman kernel decomposes as
\[
   B_{m/n}(z,w)=\sum_{j=0}^{m-1}K_j(z,w),
\]
where
\[
   K_j(z,w)=\frac{n}{m\pi^2}\,
   \frac{F_j(z,w)G_j(z,w)(z_1\overline{w_1})^j
         (z_2\overline{w_2})^{n-1-E_j}}
        {(1-z_2\overline{w_2})^2
         \big((z_2\overline{w_2})^n-(z_1\overline{w_1})^m\big)^2},
\]
\[
   F_j(z,w)=(j+1)(z_2\overline{w_2})^n
      +(m-j-1)(z_1\overline{w_1})^m,
\]
and
\[
   G_j(z,w)=j+1-\frac mn E_j+
   \left(\frac mn+\frac mnE_j-j-1\right)z_2\overline{w_2}.
\]
\end{lemma}

\begin{lemma}\label{lem:numerator-cone}
Write the kernel with a common denominator as
\[
   B_{m/n}(z,w)=
   \frac{\mathcal N(z,w)}
        {(1-z_2\overline{w_2})^2
        \big((z_2\overline{w_2})^n-(z_1\overline{w_1})^m\big)^2},
   \qquad
   \mathcal N(z,w)=\sum_{j=0}^{m-1}N_j(z,w),
\]
where
\[
   N_j(z,w)=\frac{n}{m\pi^2}F_j(z,w)G_j(z,w)
      (z_1\overline{w_1})^j(z_2\overline{w_2})^{n-1-E_j}.
\]
For every $0<\omega<\pi/2$ there exists
$\theta(\omega)>0$, depending only on $m,n$ and $\omega$, such
that the following holds.  If
\[
   z_1\overline{w_1},\ z_2\overline{w_2}
   \in\Gamma_{\theta(\omega)},
\]
then
\[
   |G_j(z,w)|\approx1,
   \qquad
   |F_j(z,w)|\gtrsim
   |z_2\overline{w_2}|^n+|z_1\overline{w_1}|^m,
   \quad 0\le j\le m-1,
\]
and $N_j(z,w)\in\Gamma_\omega$ for every $0\le j\le m-1$.  Consequently,
\begin{equation}\label{eq:cone-no-cancellation}
   |\mathcal N(z,w)|
   \ge \cos\omega\sum_{j=0}^{m-1}|N_j(z,w)|.
\end{equation}
In particular, the full numerator is nonzero on this cone-localized region.
\end{lemma}

\begin{proof}
Put $s=z_1\overline{w_1}$ and $t=z_2\overline{w_2}$.  Write
\[
   G_j(z,w)=g_j+h_j t,
   \qquad
   g_j=j+1-\frac mnE_j,
   \qquad
   h_j=\frac mn+\frac mnE_j-j-1.
\]
Since
\[
   E_j\le \frac{(j+1)n-1}{m},
\]
we have $g_j\ge1/n$, while $g_j+h_j=m/n$.  Hence
\[
   d_0=\min_{0\le j\le m-1}\min\left\{g_j,\frac mn\right\}>0,
   \qquad
   H_0=\max_{0\le j\le m-1}|h_j|<\infty.
\]
If $t=re^{i\varphi}$, then
\begin{equation}\label{eq:G-positive-axis-perturbation}
   G_j(z,w)=(1-r)g_j+r\frac mn+h_jr(e^{i\varphi}-1).
\end{equation}
Set
\[
   A_j(r)=(1-r)g_j+r\frac mn.
\]
Then
\begin{equation}\label{eq:Aj-lower-bound}
   (1-r)g_j+r\frac mn=A_j(r)\ge d_0,
   \qquad 0\le r<1,\quad 0\le j\le m-1.
\end{equation}
Choose $\theta>0$ so small that, whenever $|\varphi|<\theta$,
\begin{equation}\label{eq:G-perturb-small}
   H_0|e^{i\varphi}-1|
   < \frac{d_0}{2}
      \min\left\{1,\tan\frac{\omega}{3}\right\}.
\end{equation}
Write
\[
   G_j(z,w)=A_j(r)+E_j(r,\varphi),
   \qquad
   E_j(r,\varphi)=h_jr(e^{i\varphi}-1).
\]
By \eqref{eq:G-perturb-small},
\[
   |E_j(r,\varphi)|
   \le H_0|e^{i\varphi}-1|
   <\frac{d_0}{2}
   \le \frac{A_j(r)}{2}.
\]
Hence
\[
   \Re G_j(z,w)\ge A_j(r)-|E_j(r,\varphi)|\ge\frac{A_j(r)}2\ge\frac{d_0}{2}.
\]
If
\[
   d_1=\max_{0\le j\le m-1}\max\left\{g_j,\frac mn\right\},
\]
then $A_j(r)\le d_1$, and therefore
\[
   \frac{d_0}{2}\le |G_j(z,w)|\le d_1+\frac{d_0}{2}.
\]
This proves $|G_j(z,w)|\approx1$, uniformly in $j,z,w$.  Moreover, set
$\tau=\tan(\omega/3)$.  Since $0<\omega/3<\pi/6$, we have $0<\tau<1$.
By \eqref{eq:Aj-lower-bound} and \eqref{eq:G-perturb-small},
\[
   |E_j(r,\varphi)|<\frac{\tau}{2}A_j(r).
\]
Therefore
\[
   \left|\tan\bigl(\arg G_j(z,w)\bigr)\right|
   =\frac{|\Im G_j(z,w)|}{\Re G_j(z,w)}
   \le \frac{|E_j(r,\varphi)|}{A_j(r)-|E_j(r,\varphi)|}
   <\frac{\tau/2}{1-\tau/2}<\tau=\tan\frac{\omega}{3}.
\]
Since $\Re G_j(z,w)>0$ and $0<\omega/3<\pi/2$, it follows that
\[
   |\arg G_j(z,w)|<\frac{\omega}{3}.
\]

Next set $X=t^n$ and $Y=s^m$.  Because $z,w\in H_{m/n}$,
\begin{equation}\label{eq:Y-smaller-X}
   |Y|=|z_1|^m|w_1|^m<|z_2|^n|w_2|^n=|X|.
\end{equation}
After decreasing $\theta$, we may assume
$\max\{m,n\}\theta<\omega/3$.  Thus $X,Y\in\Gamma_{\omega/3}$.
Since this cone is convex and
\[
   F_j=(j+1)X+(m-j-1)Y
\]
has nonnegative coefficients, $F_j\in\Gamma_{\omega/3}$.  Also,
because $X,Y\in\Gamma_{\omega/3}$,
\begin{align*}
   |F_j|
   &\ge \Re F_j \\
   &=(j+1)\Re X+(m-j-1)\Re Y \\
   &\ge \cos\frac{\omega}{3}
       \big((j+1)|X|+(m-j-1)|Y|\big) \\
   &\ge \frac12\cos\frac{\omega}{3}(|X|+|Y|),
\end{align*}
where the last inequality follows from \eqref{eq:Y-smaller-X} and
\[
   (j+1)|X|+(m-j-1)|Y|
   \ge |X|
   \ge \frac12(|X|+|Y|).
\]
Finally, the exponent $n-1-E_j$ is nonnegative, and
\[
   j+n-1-E_j\le m+n-2.
\]
Choose $\theta$ still smaller so that
$(m+n-2)\theta<\omega/3$.  The monomial
$s^j t^{n-1-E_j}$ then has argument of absolute value less than $\omega/3$
whenever it is nonzero.  The three factors in $N_j$ therefore have total
argument of absolute value less than $\omega$, proving
$N_j\in\Gamma_\omega$.  We denote this final value by $\theta(\omega)$.

Because $\omega<\pi/2$, every nonzero $N_j$ has positive real part and
\[
   \Re\mathcal N(z,w)
   =\sum_j\Re N_j(z,w)
   \ge\cos\omega\sum_j|N_j(z,w)|.
\]
Since $|\mathcal N|\ge\Re\mathcal N$, this proves
\eqref{eq:cone-no-cancellation}.  To obtain nonvanishing, note that
$t=z_2\overline{w_2}\ne0$.  The preceding estimates give
$F_0(z,w)\ne0$ and $G_0(z,w)\ne0$, while
$t^{n-1-E_0}\ne0$.  Thus $N_0(z,w)\ne0$, so the right-hand side of
\eqref{eq:cone-no-cancellation} is positive and $\mathcal N(z,w)\ne0$.
\end{proof}

\begin{lemma}\label{lem:global-finite-sum}
Here $\Phi$, $R$, and $\rho$ are defined by \eqref{eq:PhiRrho}.  For
$0\le j\le m-1$, put
\[
   \kappa_j=E_j+1-\frac{nj}{m},
   \qquad
   A_j(z)=|\Phi(z)|^{j/m}|z_2|^{-\kappa_j}.
\]
Then, for all $z,w\in H_{m/n}$,
\begin{equation}\label{eq:global-offdiag-finitesum}
   |B_{m/n}(z,w)|
   \lesssim
   \frac{\sum_{j=0}^{m-1}A_j(z)A_j(w)}
        {|1-z_2\overline{w_2}|^2
        |1-\Phi(z)\overline{\Phi(w)}|^2},
\end{equation}
while on the diagonal
\begin{equation}\label{eq:global-diagonal-finitesum}
   B_{m/n}(z,z)
   \approx
   \frac{\sum_{j=0}^{m-1}A_j(z)^2}{R(z)^2\rho(z)^2}.
\end{equation}
Consequently,
\begin{equation}\label{eq:global-positive-majorant}
\begin{aligned}
&\frac{|B_{m/n}(z,w)|^{c/2}}
        {B_{m/n}(z,z)^{a/2}B_{m/n}(w,w)^{b/2}}       \\
&\qquad\lesssim
   \frac{R(z)^a\rho(z)^aR(w)^b\rho(w)^b}
        {|1-z_2\overline{w_2}|^c|1-\Phi(z)\overline{\Phi(w)}|^c}
   \frac{\left(\sum_{j=0}^{m-1}A_j(z)A_j(w)\right)^{c/2}}
        {\left(\sum_{j=0}^{m-1}A_j(z)^2\right)^{a/2}
         \left(\sum_{j=0}^{m-1}A_j(w)^2\right)^{b/2}}.
\end{aligned}
\end{equation}
\end{lemma}

\begin{proof}
Fix $j\in\{0,\ldots,m-1\}$.  Since $z,w\in H_{m/n}$,
\[
   |(z_1\overline{w_1})^m|=|z_1|^m|w_1|^m
   < |z_2|^n|w_2|^n=|z_2\overline{w_2}|^n.
\]
Therefore the numerator factor $F_j$ in Lemma~\ref{lem:subkernel-formula}
satisfies
\[
\begin{aligned}
   |F_j(z,w)|
   &\le (j+1)|z_2\overline{w_2}|^n
      +(m-j-1)|(z_1\overline{w_1})^m|  \\
   &\le m |z_2\overline{w_2}|^n.
\end{aligned}
\]
Moreover, writing
\[
   G_j(z,w)=g_j+h_j z_2\overline{w_2},
   \qquad
   g_j=j+1-\frac mnE_j,
   \quad
   h_j=\frac mn+\frac mnE_j-j-1,
\]
we have $|z_2\overline{w_2}|<1$; hence
\[
   |G_j(z,w)|\le |g_j|+|h_j|,
\]
which is a constant depending only on $m,n,j$.  Also
\[
\begin{aligned}
   (z_2\overline{w_2})^n-(z_1\overline{w_1})^m
   &=(z_2\overline{w_2})^n
     \left(1-\frac{(z_1\overline{w_1})^m}{(z_2\overline{w_2})^n}\right) \\
   &=(z_2\overline{w_2})^n
     \left(1-\Phi(z)\overline{\Phi(w)}\right).
\end{aligned}
\]
Substituting these estimates into the subkernel formula gives
\[
\begin{aligned}
   |K_j(z,w)|
   &\lesssim
   \frac{|z_2\overline{w_2}|^n
          |z_1\overline{w_1}|^j
          |z_2\overline{w_2}|^{n-1-E_j}}
        {|1-z_2\overline{w_2}|^2
         |z_2\overline{w_2}|^{2n}
         |1-\Phi(z)\overline{\Phi(w)}|^2}  \\
   &=
   \frac{|z_1\overline{w_1}|^j
          |z_2\overline{w_2}|^{-1-E_j}}
        {|1-z_2\overline{w_2}|^2
         |1-\Phi(z)\overline{\Phi(w)}|^2}.
\end{aligned}
\]
Since
\[
   |z_1|^j=|\Phi(z)|^{j/m}|z_2|^{nj/m},
   \qquad
   |w_1|^j=|\Phi(w)|^{j/m}|w_2|^{nj/m},
\]
it follows that
\[
\begin{aligned}
   |K_j(z,w)|
   &\lesssim
   \frac{|\Phi(z)|^{j/m}|\Phi(w)|^{j/m}
          |z_2\overline{w_2}|^{nj/m-1-E_j}}
        {|1-z_2\overline{w_2}|^2
         |1-\Phi(z)\overline{\Phi(w)}|^2} \\
   &=
   \frac{|\Phi(z)|^{j/m}|z_2|^{-\kappa_j}
          |\Phi(w)|^{j/m}|w_2|^{-\kappa_j}}
        {|1-z_2\overline{w_2}|^2
         |1-\Phi(z)\overline{\Phi(w)}|^2} \\
   &=
   \frac{A_j(z)A_j(w)}
        {|1-z_2\overline{w_2}|^2
         |1-\Phi(z)\overline{\Phi(w)}|^2}.
\end{aligned}
\]
Summing over $j$ gives \eqref{eq:global-offdiag-finitesum}.

We next prove the diagonal estimate.  Put $w=z$ in the preceding subkernel
formula.  Then
\[
   F_j(z,z)=(j+1)|z_2|^{2n}+(m-j-1)|z_1|^{2m}.
\]
Since $|z_1|^{2m}<|z_2|^{2n}$, we have
\[
   (j+1)|z_2|^{2n}
   \le F_j(z,z)
   \le m |z_2|^{2n},
\]
and hence $F_j(z,z)\approx |z_2|^{2n}$.  For $G_j$, set
$t=|z_2|^2\in(0,1)$.  With the constants $g_j,h_j$ defined above,
\[
   G_j(z,z)=g_j+h_j t.
\]
The definition
\[
   E_j=\left\lfloor \frac{(j+1)n-1}{m}\right\rfloor
\]
implies
\[
   E_j\le \frac{(j+1)n-1}{m},
   \qquad\text{and hence}\qquad
   g_j=j+1-\frac mnE_j\ge \frac1n>0.
\]
Furthermore,
\[
   g_j+h_j=\frac mn.
\]
Thus
\[
   G_j(z,z)=(1-t)g_j+t\frac mn,
\]
so $G_j(z,z)$ lies between the two positive constants
$\min\{g_j,m/n\}$ and $\max\{g_j,m/n\}$.  Consequently
$G_j(z,z)\approx1$, with constants depending only on $m,n,j$.

Finally,
\[
\begin{aligned}
   &(z_2\overline{z_2})^n-(z_1\overline{z_1})^m  \\
   &\qquad = |z_2|^{2n}
      \left(1-\frac{|z_1|^{2m}}{|z_2|^{2n}}\right)
      = |z_2|^{2n}R(z),
\end{aligned}
\]
and $1-z_2\overline{z_2}=\rho(z)$.  Combining these identities with
$F_j(z,z)\approx |z_2|^{2n}$ and $G_j(z,z)\approx1$, we obtain
\[
\begin{aligned}
   K_j(z,z)
   &\approx
   \frac{|z_2|^{2n}|z_1|^{2j}|z_2|^{2n-2-2E_j}}
        {\rho(z)^2 |z_2|^{4n}R(z)^2}   \\
   &=
   \frac{|\Phi(z)|^{2j/m}|z_2|^{2nj/m-2-2E_j}}
        {R(z)^2\rho(z)^2}  \\
   &=
   \frac{|\Phi(z)|^{2j/m}|z_2|^{-2\kappa_j}}
        {R(z)^2\rho(z)^2}
    =\frac{A_j(z)^2}{R(z)^2\rho(z)^2}.
\end{aligned}
\]
Since the spaces in the Edholm--McNeal decomposition are orthogonal,
$B_{m/n}(z,z)=\sum_{j=0}^{m-1}K_j(z,z)$, and
\eqref{eq:global-diagonal-finitesum} follows.  Finally,
\eqref{eq:global-positive-majorant} is obtained by raising
\eqref{eq:global-offdiag-finitesum} to the power $c/2$ and using
\eqref{eq:global-diagonal-finitesum} for the two diagonal factors.  The diagonal
estimate is two-sided and $B_{m/n}(\xi,\xi)>0$, so the division by the diagonal
powers is legitimate for all real $a,b$, including negative values; the factors
$R(z)^a$ and $R(w)^b$ in \eqref{eq:global-positive-majorant} are therefore valid
without any sign restriction on $a$ or $b$.
\end{proof}

\begin{lemma}\label{lem:critical-diagonal}
For every fixed $0<\delta<1$,
\begin{equation}\label{eq:critical-diagonal}
   B_{m/n}(z,z)
   \approx
   \frac{|z_2|^{-2\kappa}}{R(z)^2\rho(z)^2},
   \qquad
   \delta\le \frac{|z_1|^m}{|z_2|^n}<1.
\end{equation}
Here $R(z)=1-|z_1|^{2m}/|z_2|^{2n}$ and $\rho(z)=1-|z_2|^2$.
The implicit constants may depend on $\delta$, but not on $z$.
\end{lemma}

\begin{proof}
We derive this from Lemma~\ref{lem:global-finite-sum}.  First, for every
$0\le j\le m-1$,
\[
   \kappa_j=E_j+1-\frac{nj}{m}
   \le \frac{(j+1)n-1}{m}+1-\frac{nj}{m}
   =\frac{m+n-1}{m}=\kappa.
\]
Since $0\le |\Phi(z)|<1$ and $0<|z_2|<1$, we have
\[
   A_j(z)=|\Phi(z)|^{j/m}|z_2|^{-\kappa_j}
   \le |z_2|^{-\kappa}.
\]
Hence
\[
   \sum_{j=0}^{m-1}A_j(z)^2\lesssim |z_2|^{-2\kappa}.
\]
By \eqref{eq:global-diagonal-finitesum},
\[
   B_{m/n}(z,z)\lesssim
   \frac{|z_2|^{-2\kappa}}{R(z)^2\rho(z)^2}.
\]
It remains to prove the reverse inequality under the hypothesis
$\delta\le |\Phi(z)|<1$.  Since $\gcd(m,n)=1$, multiplication by $n$
permutes the residue classes modulo $m$.  Hence there is a unique
$j_*\in\{0,\ldots,m-1\}$ such that
\[
   n(j_*+1)\equiv 1 \pmod m.
\]
Then
\[
   E_{j_*}=\frac{n(j_*+1)-1}{m},
   \qquad
   \kappa_{j_*}=E_{j_*}+1-\frac{nj_*}{m}
   =\frac{m+n-1}{m}=\kappa.
\]
Therefore, if $\delta\le |\Phi(z)|<1$, then
\[
   \sum_{j=0}^{m-1}A_j(z)^2
   \ge A_{j_*}(z)^2
   = |\Phi(z)|^{2j_*/m}|z_2|^{-2\kappa}
   \gtrsim |z_2|^{-2\kappa},
\]
where the implicit constant in this lower bound depends on $\delta$.  The
lower bound in \eqref{eq:critical-diagonal} follows again from
\eqref{eq:global-diagonal-finitesum}.
\end{proof}

\begin{lemma}\label{lem:power-substitution}
Fix $\xi\in\Dstar$.
\begin{enumerate}[label=\textup{(\roman*)}]
\item For every nonnegative measurable function $F$ on $\D$,
\begin{equation}\label{eq:exact-power-substitution}
   \int_{|w_1|^m<|\xi|^n}
      F\!\left(\frac{w_1^m}{\xi^n}\right)\dA(w_1)
   =\frac{|\xi|^{2n/m}}{m}
      \int_\D F(u)|u|^{2/m-2}\dA(u).
\end{equation}
\item Let $V\subset\Dstar$ lie in a simply connected region on which a branch
of $u^{1/m}$ is fixed, and choose one value of $\xi^{n/m}$.  If
\[
   T_\xi(u)=\xi^{n/m}u^{1/m},\qquad u\in V,
\]
then $T_\xi$ is one-to-one and
\begin{equation}\label{eq:one-sheet-power-substitution}
   \int_{T_\xi(V)}
      F\!\left(\frac{w_1^m}{\xi^n}\right)\dA(w_1)
   =\frac{|\xi|^{2n/m}}{m^2}
      \int_V F(u)|u|^{2/m-2}\dA(u).
\end{equation}
\end{enumerate}
Thus the full-fiber formula contains the multiplicity factor $1/m$, whereas
the change of variables in \textup{(ii)} has Jacobian factor $1/m^2$.
\end{lemma}

\begin{proof}
Consider the proper map
\[
   \Psi_\xi(w_1)=\frac{w_1^m}{\xi^n},
   \qquad |w_1|<|\xi|^{n/m}.
\]
For $u\in\Dstar$, the equation $\Psi_\xi(w_1)=u$ has exactly $m$ solutions,
all of modulus $|\xi|^{n/m}|u|^{1/m}$.  Since
\[
   \Psi_\xi'(w_1)=\frac{mw_1^{m-1}}{\xi^n},
\]
the area formula gives
\begin{align*}
   \sum_{\Psi_\xi(w_1)=u}\frac{1}{|\Psi_\xi'(w_1)|^2}
   &=m\frac{|\xi|^{2n}}
      {m^2\bigl(|\xi|^{n/m}|u|^{1/m}\bigr)^{2m-2}} \\
   &=\frac{|\xi|^{2n/m}}{m}|u|^{2/m-2}.
\end{align*}
Integrating against $F(u)\dA(u)$ proves
\eqref{eq:exact-power-substitution}.

For the branch fixed in \textup{(ii)},
\[
   T_\xi'(u)=\frac{\xi^{n/m}}{m}u^{1/m-1},
\]
so
\[
   |T_\xi'(u)|^2
   =\frac{|\xi|^{2n/m}}{m^2}|u|^{2/m-2}.
\]
The change-of-variables theorem yields
\eqref{eq:one-sheet-power-substitution}.
\end{proof}

\begin{lemma}\label{lem:disc-estimate}
Let $\mu>-1$, $\nu>-2$, $\lambda>0$, and let $A,\omega\in\R$.  Assume
\begin{equation}\label{eq:three-case-condition}
   A>\omega,
   \qquad
   A+\mu+2-\lambda\ge \omega.
\end{equation}
Then, uniformly for $\zeta\in\D$,
\begin{equation}\label{eq:three-case-consequence}
   (1-|\zeta|^2)^A\int_\D
   \frac{(1-|u|^2)^\mu |u|^\nu}
        {|1-\zeta\overline u|^\lambda}\dA(u)
   \lesssim (1-|\zeta|^2)^\omega.
\end{equation}
\end{lemma}

\begin{proof}
Set
\[
   I(\zeta)=\int_\D
   \frac{(1-|u|^2)^\mu |u|^\nu}
        {|1-\zeta\overline u|^\lambda}\dA(u),
   \qquad X=1-|\zeta|^2.
\]
We claim that
\begin{equation}\label{eq:direct-disc-three-cases}
   I(\zeta)\lesssim
   \begin{cases}
      1, & \mu+2-\lambda>0,\\[1mm]
      \log\dfrac{e}{X}, & \mu+2-\lambda=0,\\[2mm]
      X^{\mu+2-\lambda}, & \mu+2-\lambda<0.
   \end{cases}
\end{equation}
On $|u|\le1/2$, the denominator is bounded below and
$\int_{|u|\le1/2}|u|^\nu\dA(u)<\infty$ because $\nu>-2$.  On
$|u|>1/2$, one has $|u|^\nu\approx1$, with constants depending only on
$\nu$.  By rotation invariance assume $\zeta=s\in[0,1)$ and write
$u=re^{i\vartheta}$.  Uniformly for
$1/2<r<1$ and $|\vartheta|\le\pi$,
\[
   1-r^2\approx1-r,
   \qquad
   |1-sre^{-i\vartheta}|\approx(1-sr)+|\vartheta|.
\]
For the second comparison, put $a=sr$.  The upper bound follows from
\[
   |1-ae^{-i\vartheta}|
   \le (1-a)+a|1-e^{-i\vartheta}|
   \lesssim (1-a)+|\vartheta|.
\]
For the lower bound, if $a\le1/2$, then
$|1-ae^{-i\vartheta}|\ge1-a\ge1/2$, whereas
$(1-a)+|\vartheta|\le1+\pi$.  If $a>1/2$, then
\[
   |1-ae^{-i\vartheta}|^2
   =(1-a)^2+2a(1-\cos\vartheta)
   \gtrsim (1-a)^2+|\vartheta|^2,
\]
because $1-\cos\vartheta\approx \vartheta^2$ for $|\vartheta|\le\pi$.
Hence $|1-ae^{-i\vartheta}|\gtrsim (1-a)+|\vartheta|$, and the comparison
follows.
It follows that the remaining part of $I(\zeta)$ is bounded by a constant
multiple of
\[
   \int_{1/2}^1(1-r)^\mu
      \left[\int_{-\pi}^{\pi}
      \frac{d\vartheta}{((1-sr)+|\vartheta|)^\lambda}\right]dr.
\]
The bracketed integral is bounded by
\[
   \begin{cases}
      1, & 0<\lambda<1,\\
      \log\dfrac{e}{1-sr}, & \lambda=1,\\[2mm]
      (1-sr)^{1-\lambda}, & \lambda>1.
   \end{cases}
\]
Put $t=1-r$ and $\delta=1-s$.  Since
$1-sr=\delta+st\approx\delta+t$ on the range under consideration, it remains
to estimate
\[
   \int_0^1t^\mu\Psi_\lambda(\delta+t)\,dt,
\]
where $\Psi_\lambda(y)$ is respectively $1$, $\log(e/y)$, or
$y^{1-\lambda}$.  The cases $\lambda<1$ and $\lambda=1$ are uniformly bounded;
for the latter, use
\[
   \log\frac{e}{\delta+t}\le\log\frac{e}{t},
   \qquad
   \int_0^1t^\mu\log\frac{e}{t}\,dt<\infty.
\]
If $\lambda>1$, split at $t=\delta$:
\begin{align*}
   \int_0^1t^\mu(\delta+t)^{1-\lambda}\,dt
   &\lesssim
   \delta^{1-\lambda}\int_0^\delta t^\mu\,dt
   +\int_\delta^1t^{\mu+1-\lambda}\,dt  \\
   &\lesssim
   \begin{cases}
      1, & \mu+2-\lambda>0,\\[1mm]
      \log\dfrac{e}{\delta}, & \mu+2-\lambda=0,\\[2mm]
      \delta^{\mu+2-\lambda}, & \mu+2-\lambda<0.
   \end{cases}
\end{align*}
Because $\delta\le X=(1-s)(1+s)\le2\delta$, this proves
\eqref{eq:direct-disc-three-cases}.

We now multiply by $X^A$.  If $\lambda<\mu+2$, then
$X^A I(\zeta)\lesssim X^A\le X^\omega$ because $A>\omega$.  If
$\lambda=\mu+2$, then
\[
   X^A I(\zeta)
   \lesssim X^\omega
      \left(X^{A-\omega}\log\frac{e}{X}\right)
   \lesssim X^\omega.
\]
If $\lambda>\mu+2$, then
\[
   X^A I(\zeta)\lesssim X^{A+\mu+2-\lambda}
   \le X^\omega
\]
by the second assumption in \eqref{eq:three-case-condition}.  This proves
\eqref{eq:three-case-consequence}.
\end{proof}

\begin{lemma}\cite[Lemma~5.2]{Liu2015}\label{lem:schur}
Let $(X,\mu)$ be a measure space, $1<p<\infty$, and $q=p/(p-1)$.  Let
$T\ge0$ be a measurable function on $X\times X$.  If there is a positive
measurable function $h$ such that
\[
   \int_X T(x,y)h(y)^q\,d\mu(y)\lesssim h(x)^q,
   \qquad
   \int_X T(x,y)h(x)^p\,d\mu(x)\lesssim h(y)^p,
\]
then the integral operator with kernel $T$ is bounded on $L^p(X,d\mu)$.
\end{lemma}


\section{Sufficiency}

\begin{proposition}\label{prop:sufficiency}
If $\mathcal C(a,b,c;p)$ holds, then $S^+_{a,b,c}$ is bounded on
$L^p(H_{m/n})$.
\end{proposition}

\begin{proof}
Let $q=p/(p-1)$.  For $0\le j\le m-1$, retain the quantities
$A_j$ and $\kappa_j$ from Lemma~\ref{lem:global-finite-sum}, and set
\[
   P(z)=\sum_{j=0}^{m-1}A_j(z)^2,
   \qquad
   Q(z,w)=\sum_{j=0}^{m-1}A_j(z)A_j(w).
\]
Denote the kernel of $S^+_{a,b,c}$ by
\[
   \mathcal K(z,w)=
   \frac{|B_{m/n}(z,w)|^{c/2}}
        {B_{m/n}(z,z)^{a/2}B_{m/n}(w,w)^{b/2}}.
\]
Lemma~\ref{lem:global-finite-sum} gives
\begin{equation}\label{eq:schur-finite-sum}
\begin{aligned}
   \mathcal K(z,w)
   \lesssim{}&
   \frac{R(z)^a\rho(z)^aR(w)^b\rho(w)^b}
        {|1-z_2\overline{w_2}|^c
         |1-\Phi(z)\overline{\Phi(w)}|^c} \\
   &\times
   \frac{Q(z,w)^{c/2}}
        {P(z)^{a/2}P(w)^{b/2}}.
\end{aligned}
\end{equation}
By Cauchy--Schwarz, $Q(z,w)\le P(z)^{1/2}P(w)^{1/2}$, and hence
\begin{equation}\label{eq:schur-P-kernel}
   \mathcal K(z,w)
   \lesssim
   \frac{R(z)^a\rho(z)^aR(w)^b\rho(w)^b
          P(z)^{(c-2a)/4}P(w)^{(c-2b)/4}}
        {|1-z_2\overline{w_2}|^c
         |1-\Phi(z)\overline{\Phi(w)}|^c}.
\end{equation}
This estimate is valid globally and does not require
$|\Phi(z)|$ or $|\Phi(w)|$ to be bounded away from zero.

For all $j$, $\kappa_j\le\kappa$.  Since $|\Phi(z)|<1$ and $|z_2|<1$,
\[
   A_j(z)\le |z_2|^{-\kappa},
   \qquad
   P(z)\lesssim |z_2|^{-2\kappa}.
\]
On the other hand, $A_0(z)=|z_2|^{-\kappa_0}\ge1$, so $P(z)\ge1$.
Consequently, for every $t\in\R$,
\begin{equation}\label{eq:P-real-power}
   P(z)^t\lesssim |z_2|^{-2\kappa t_+},
   \qquad t_+=\max\{t,0\}.
\end{equation}
Recalling that $\kappa=(m+n-1)/m$, set
\[
   \ell_a=\frac\kappa2(c-2a)_+,
   \qquad
   \ell_b=\frac\kappa2(c-2b)_+.
\]
Equations \eqref{eq:schur-P-kernel} and \eqref{eq:P-real-power} imply
\begin{equation}\label{eq:schur-vertex-kernel}
   \mathcal K(z,w)
   \lesssim
   \frac{R(z)^a\rho(z)^aR(w)^b\rho(w)^b
          |z_2|^{-\ell_a}|w_2|^{-\ell_b}}
        {|1-z_2\overline{w_2}|^c
         |1-\Phi(z)\overline{\Phi(w)}|^c}.
\end{equation}

We next choose a Schur weight.  Condition \eqref{cond:boundary} is equivalent to
\[
   a>-\frac1p,
   \qquad
   b>-\frac1q.
\]
Consider the intervals
\[
   I_q=\left(-\frac{b+1}{q},\frac aq\right),
   \qquad
   I_p=\left(-\frac{a+1}{p},\frac bp\right).
\]
They have nonempty intersection.  Indeed, $a+b+1>0$ follows from the two
preceding inequalities, so each interval is nonempty; moreover,
\[
   -\frac{b+1}{q}<\frac bp
   \quad\Longleftrightarrow\quad b>-\frac1q,
   \qquad
   -\frac{a+1}{p}<\frac aq
   \quad\Longleftrightarrow\quad a>-\frac1p.
\]
Choose $\sigma_1,\sigma_2\in I_q\cap I_p$.  Then, for $i=1,2$,
\begin{equation}\label{eq:sigma-conditions}
   b+q\sigma_i>-1,
   \quad q\sigma_i<a,
   \quad a+p\sigma_i>-1,
   \quad p\sigma_i<b.
\end{equation}

Put
\[
   d_0=\frac{2n}{m}+2=\frac{2(m+n)}m.
\]
The two inequalities in \eqref{cond:vertex} are equivalent to
\begin{equation}\label{eq:vertex-loss-conditions}
   p\ell_a<d_0,
   \qquad
   q\ell_b<d_0.
\end{equation}
When $c-2a>0$, the first relation is exactly
$(c-2a)(m+n-1)p<4(m+n)$. If $c-2a\le0$, both statements are automatic.
Similarly, the second relation follows by writing the right-hand
inequality in \eqref{cond:vertex} as
$q(c-2b)(m+n-1)<4(m+n)$.

We claim that
\begin{equation}\label{eq:tau-interval}
   \left(\frac{-d_0+\ell_b}{q},-\frac{\ell_a}{q}\right]
   \cap
   \left(\frac{-d_0+\ell_a}{p},-\frac{\ell_b}{p}\right]
\end{equation}
is nonempty.  From \eqref{eq:vertex-loss-conditions},
$\ell_a<d_0/p$ and $\ell_b<d_0/q$, hence
$\ell_a+\ell_b<d_0$.  Thus each interval in \eqref{eq:tau-interval} is
nonempty.  The two cross inequalities are precisely
\[
   \frac{-d_0+\ell_b}{q}< -\frac{\ell_b}{p}
   \quad\Longleftrightarrow\quad q\ell_b<d_0,
\]
and
\[
   \frac{-d_0+\ell_a}{p}< -\frac{\ell_a}{q}
   \quad\Longleftrightarrow\quad p\ell_a<d_0.
\]
Choose $\tau$ in the intersection.  Equivalently,
\begin{align}
   \frac{2n}{m}+q\tau-\ell_b&>-2,
   &q\tau&\le-\ell_a,\label{eq:tau-q}\\
   \frac{2n}{m}+p\tau-\ell_a&>-2,
   &p\tau&\le-\ell_b.\label{eq:tau-p}
\end{align}
Set
\[
   h(z)=R(z)^{\sigma_1}\rho(z)^{\sigma_2}|z_2|^\tau.
\]

We prove the first Schur inequality.  From
\eqref{eq:schur-vertex-kernel},
\begin{equation}\label{eq:first-schur-start}
\begin{aligned}
&\int_{H_{m/n}}\mathcal K(z,w)h(w)^q\dV(w) \\
&\quad\lesssim
R(z)^a\rho(z)^a|z_2|^{-\ell_a}
\int_{H_{m/n}}
\frac{R(w)^{b+q\sigma_1}\rho(w)^{b+q\sigma_2}
      |w_2|^{q\tau-\ell_b}}
     {|1-z_2\overline{w_2}|^c
      |1-\Phi(z)\overline{\Phi(w)}|^c}
\dV(w).
\end{aligned}
\end{equation}
For fixed $w_2\in\Dstar$, apply
Lemma~\ref{lem:power-substitution}\textup{(i)} with $u=\Phi(w)$.  The integral on the right of \eqref{eq:first-schur-start} equals $1/m$
times the following product integral; the fixed factor $1/m$ is absorbed into
the implicit constant:
\begin{equation}\label{eq:first-schur-product}
\begin{aligned}
&\int_{\Dstar}
   \frac{(1-|w_2|^2)^{b+q\sigma_2}
          |w_2|^{2n/m+q\tau-\ell_b}}
        {|1-z_2\overline{w_2}|^c} \\
&\qquad\times
   \left[
   \int_\D
   \frac{(1-|u|^2)^{b+q\sigma_1}|u|^{2/m-2}}
        {|1-\Phi(z)\overline u|^c}\dA(u)
   \right]\dA(w_2).
\end{aligned}
\end{equation}
The exponent $2/m-2$ is strictly greater than $-2$, exactly as required at
$u=0$.

Apply Lemma~\ref{lem:disc-estimate} to the inner integral with
\[
   \mu=b+q\sigma_1,
   \quad \nu=\frac2m-2,
   \quad \lambda=c,
   \quad A=a,
   \quad \omega=q\sigma_1.
\]
The first and second inequalities in \eqref{eq:sigma-conditions} give
$\mu>-1$ and $A>\omega$, while \eqref{cond:pinch} gives
\[
   A+\mu+2-\lambda
   =q\sigma_1+(a+b+2-c)\ge q\sigma_1.
\]
Therefore
\begin{equation}\label{eq:first-inner-disc}
   R(z)^a
   \int_\D
   \frac{(1-|u|^2)^{b+q\sigma_1}|u|^{2/m-2}}
        {|1-\Phi(z)\overline u|^c}\dA(u)
   \lesssim R(z)^{q\sigma_1}.
\end{equation}
Apply Lemma~\ref{lem:disc-estimate} to the $w_2$-integral with
\[
   \mu=b+q\sigma_2,
   \quad \nu=\frac{2n}{m}+q\tau-\ell_b,
   \quad \lambda=c,
   \quad A=a,
   \quad \omega=q\sigma_2.
\]
Its hypotheses follow from \eqref{eq:sigma-conditions},
\eqref{eq:tau-q}, and \eqref{cond:pinch}.  Hence
\begin{equation}\label{eq:first-outer-disc}
   \rho(z)^a
   \int_\D
   \frac{(1-|w_2|^2)^{b+q\sigma_2}
          |w_2|^{2n/m+q\tau-\ell_b}}
        {|1-z_2\overline{w_2}|^c}\dA(w_2)
   \lesssim\rho(z)^{q\sigma_2}.
\end{equation}
Finally, $q\tau\le-\ell_a$ and $0<|z_2|<1$ imply
$|z_2|^{-\ell_a}\le|z_2|^{q\tau}$.  Combining
\eqref{eq:first-schur-start}--\eqref{eq:first-outer-disc} yields
\begin{equation}\label{eq:first-schur-final}
   \int_{H_{m/n}}\mathcal K(z,w)h(w)^q\dV(w)
   \lesssim h(z)^q.
\end{equation}

We now verify the second Schur inequality.  Interchanging the
roles of $z$ and $w$ in \eqref{eq:schur-vertex-kernel} gives
\begin{equation}\label{eq:second-schur-start}
\begin{aligned}
&\int_{H_{m/n}}\mathcal K(z,w)h(z)^p\dV(z) \\
&\quad\lesssim
R(w)^b\rho(w)^b|w_2|^{-\ell_b}
\int_{H_{m/n}}
\frac{R(z)^{a+p\sigma_1}\rho(z)^{a+p\sigma_2}
      |z_2|^{p\tau-\ell_a}}
     {|1-z_2\overline{w_2}|^c
      |1-\Phi(z)\overline{\Phi(w)}|^c}
\dV(z).
\end{aligned}
\end{equation}
Use Lemma~\ref{lem:power-substitution}\textup{(i)} with $v=\Phi(z)$. The
integral on the right of \eqref{eq:second-schur-start} equals $1/m$ times
\begin{equation}\label{eq:second-schur-product}
\begin{aligned}
&\int_{\Dstar}
   \frac{(1-|z_2|^2)^{a+p\sigma_2}
          |z_2|^{2n/m+p\tau-\ell_a}}
        {|1-z_2\overline{w_2}|^c} \\
&\qquad\times
   \left[
   \int_\D
   \frac{(1-|v|^2)^{a+p\sigma_1}|v|^{2/m-2}}
        {|1-v\overline{\Phi(w)}|^c}\dA(v)
   \right]\dA(z_2).
\end{aligned}
\end{equation}
Apply Lemma~\ref{lem:disc-estimate} to the $v$-integral with
\[
   (\mu,\nu,\lambda,A,\omega)
   =\left(a+p\sigma_1,\frac2m-2,c,b,p\sigma_1\right),
\]
and to the $z_2$-integral with
\[
   (\mu,\nu,\lambda,A,\omega)
   =\left(a+p\sigma_2,
      \frac{2n}{m}+p\tau-\ell_a,c,b,p\sigma_2\right).
\]
The required inequalities are respectively
\eqref{eq:sigma-conditions}, \eqref{eq:tau-p}, and
\eqref{cond:pinch}.  Hence
\begin{equation}\label{eq:second-inner-disc}
   R(w)^b
   \int_\D
   \frac{(1-|v|^2)^{a+p\sigma_1}|v|^{2/m-2}}
        {|1-v\overline{\Phi(w)}|^c}\dA(v)
   \lesssim R(w)^{p\sigma_1},
\end{equation}
and
\begin{equation}\label{eq:second-outer-disc}
   \rho(w)^b
   \int_\D
   \frac{(1-|z_2|^2)^{a+p\sigma_2}
          |z_2|^{2n/m+p\tau-\ell_a}}
        {|1-z_2\overline{w_2}|^c}\dA(z_2)
   \lesssim \rho(w)^{p\sigma_2}.
\end{equation}
Since $p\tau\le-\ell_b$,
$|w_2|^{-\ell_b}\le|w_2|^{p\tau}$, and therefore
\begin{equation}\label{eq:second-schur-final}
   \int_{H_{m/n}}\mathcal K(z,w)h(z)^p\dV(z)
   \lesssim h(w)^p.
\end{equation}
Equations \eqref{eq:first-schur-final} and
\eqref{eq:second-schur-final} are the hypotheses of
Lemma~\ref{lem:schur}.  Thus $S^+_{a,b,c}$ is bounded on
$L^p(H_{m/n})$.
\end{proof}

\section{Necessary conditions}

\begin{lemma}\label{lem:conical-lower}
There exists $\theta_*>0$, depending only on $m,n$, such that
$\theta_*<\pi/(10(m+n))$ and the following statements hold for every
$0<\theta_0\le\theta_*$.  Suppose
\[
   z_1\overline{w_1}\in\Gamma_{\theta_0},
   \qquad
   z_2\overline{w_2}\in\Gamma_{\theta_0}.
\]
\begin{enumerate}[label=\textup{(\roman*)}]
\item For every fixed $0<\delta<1/2$, if
\[
   \delta |z_2\overline{w_2}|^n
   \le |z_1\overline{w_1}|^m
   \le (1-\delta)|z_2\overline{w_2}|^n,
\]
then
\begin{equation}\label{eq:conical-critical-lower}
   |B_{m/n}(z,w)|
   \gtrsim
   \frac{|z_2\overline{w_2}|^{-\kappa}}
        {|1-z_2\overline{w_2}|^2
        \left|1-\dfrac{(z_1\overline{w_1})^m}
        {(z_2\overline{w_2})^n}\right|^2}.
\end{equation}
\item For every fixed $0<r<R<1$, if
\[
   r\le |z_2\overline{w_2}|\le R,
\]
then
\begin{equation}\label{eq:conical-pinch-lower}
   |B_{m/n}(z,w)|
   \gtrsim
   \left|1-\frac{(z_1\overline{w_1})^m}
   {(z_2\overline{w_2})^n}\right|^{-2}.
\end{equation}
\end{enumerate}
\end{lemma}

\begin{proof}
Apply Lemma~\ref{lem:numerator-cone} with $\omega=\pi/3$, and choose
\[
   0<\theta_*\le\theta(\pi/3)
\]
small enough that $\theta_*<\pi/(10(m+n))$.  If the two coordinate products lie
in $\Gamma_{\theta_0}$, where $0<\theta_0\le\theta_*$, then
\begin{equation}\label{eq:no-cancellation-half}
   \left|\sum_{j=0}^{m-1}N_j(z,w)\right|
   \ge\frac12\sum_{j=0}^{m-1}|N_j(z,w)|.
\end{equation}
Thus the full kernel is bounded from below by any one chosen numerator term.

For \textup{(i)}, let $j_*$ be the unique index in
$\{0,\ldots,m-1\}$ satisfying
\[
   n(j_*+1)\equiv1\pmod m.
\]
Then
\[
   E_{j_*}=\frac{n(j_*+1)-1}{m},
   \qquad
   \frac{nj_*}{m}-E_{j_*}-1=-\kappa,
\]
and
\[
   G_{j_*}(z,w)=\frac1n+\frac{m-1}{n}z_2\overline{w_2}.
\]
Lemma~\ref{lem:numerator-cone} gives
$|G_{j_*}(z,w)|\gtrsim1$ and
\[
   |F_{j_*}(z,w)|
   \gtrsim |z_2\overline{w_2}|^n
             +|z_1\overline{w_1}|^m.
\]
Under the ratio hypothesis,
\[
   |F_{j_*}(z,w)|\gtrsim |z_2\overline{w_2}|^n,
   \qquad
   |z_1\overline{w_1}|^{j_*}
   \approx |z_2\overline{w_2}|^{nj_*/m}.
\]
Using \eqref{eq:no-cancellation-half} and the common denominator, we obtain
\begin{align*}
   |B_{m/n}(z,w)|
   &\gtrsim
   \frac{|z_1\overline{w_1}|^{j_*}
          |z_2\overline{w_2}|^{2n-1-E_{j_*}}}
        {|1-z_2\overline{w_2}|^2
         |(z_2\overline{w_2})^n-(z_1\overline{w_1})^m|^2}\\
   &\approx
   \frac{|z_2\overline{w_2}|^{nj_*/m-E_{j_*}-1}}
        {|1-z_2\overline{w_2}|^2
         \left|1-\dfrac{(z_1\overline{w_1})^m}
         {(z_2\overline{w_2})^n}\right|^2},
\end{align*}
which is \eqref{eq:conical-critical-lower} because the exponent equals
$-\kappa$.

For \textup{(ii)}, use $N_0$.  Lemma~\ref{lem:numerator-cone} gives
$|G_0(z,w)|\gtrsim1$ and
\[
   |F_0(z,w)|
   =\left|(z_2\overline{w_2})^n
        +(m-1)(z_1\overline{w_1})^m\right|
   \gtrsim |z_2\overline{w_2}|^n.
\]
When $r\le|z_2\overline{w_2}|\le R$, all remaining powers of
$|z_2\overline{w_2}|$ in $N_0$ are bounded above and below by positive
constants.  Hence \eqref{eq:no-cancellation-half} gives
\[
   |B_{m/n}(z,w)|
   \gtrsim
   \frac{1}
        {|1-z_2\overline{w_2}|^2
         |(z_2\overline{w_2})^n-(z_1\overline{w_1})^m|^2}.
\]
On the same region, $|1-z_2\overline{w_2}|\approx1$ and
$|z_2\overline{w_2}|^n\approx1$, with constants depending on $r,R$.
Factoring out $(z_2\overline{w_2})^n$ proves
\eqref{eq:conical-pinch-lower}.
\end{proof}

\begin{lemma}\label{lem:fixed-test}
Let
$a,b\in\R$ and $c>0$.  Fix
\[
   0<\theta_0\le\theta_*,\qquad
   0<r_*<1,\qquad
   0<u_-<u_+<1,\qquad
   0<t_-<t_+<1.
\]
There exist a measurable set $E\Subset H_{m/n}$ with $|E|>0$ and open
intervals $J_2,J_\Phi\subset\R$ such that the following holds. Every
$w\in E$ can be written as
\[
   w=\bigl(w_2^{n/m}u^{1/m},w_2\bigr)
\]
with fixed choices of the fractional powers. If
\[
\begin{gathered}
   z=\bigl(z_2^{n/m}v^{1/m},z_2\bigr),\qquad
   v=\Phi(z)=\frac{z_1^m}{z_2^n},\qquad r_*<|v|<1,\\
   \arg z_2\in J_2,\qquad \arg v\in J_\Phi,
\end{gathered}
\]
then, for every $w\in E$ represented as above,
\[
   z_1\overline{w_1},\quad z_2\overline{w_2},\quad
   v\overline u\in\Gamma_{\theta_0},
\]
and
\begin{equation}\label{eq:fixed-test-lower}
   S^+_{a,b,c}\chi_E(z)
   \gtrsim B_{m/n}(z,z)^{-a/2}|z_2|^{-c\kappa/2}.
\end{equation}
The implicit constant may depend on the displayed fixed parameters and on
$a,b,c,m,n$, but it is independent of $z$.
\end{lemma}

\begin{proof}
Choose
\[
   0<\omega<\min\left\{\frac{\theta_0}{4},
   \frac{m\theta_0}{4(n+1)}\right\}
\]
and set
\[
   I_2=J_2=(-\omega,\omega),
   \qquad I_\Phi=J_\Phi=(-\omega,\omega).
\]
On these argument intervals use the single-valued branches
\[
   \lambda^{n/m}=|\lambda|^{n/m}
      e^{\frac nm i\arg\lambda},
   \qquad
   \lambda^{1/m}=|\lambda|^{1/m}
      e^{\frac1m i\arg\lambda}.
\]
Define
\[
\begin{aligned}
   E=\{&(w_2^{n/m}u^{1/m},w_2):
   t_-<|w_2|<t_+,
   \ \arg w_2\in I_2,\\
   &\hspace{34mm}
   u_-<|u|<u_+,
   \ \arg u\in I_\Phi\}.
\end{aligned}
\]
By \eqref{eq:one-sheet-power-substitution}, the parametrization has a
strictly positive Jacobian on a parameter set of positive measure, and hence
$|E|>0$. The strict radial bounds give $E\Subset H_{m/n}$.

Let $z$ be as in the statement and write
$w=(w_2^{n/m}u^{1/m},w_2)\in E$.  The argument restrictions imply
\[
   |\arg(z_2\overline{w_2})|<2\omega<\theta_0
\]
and
\begin{align*}
   |\arg(z_1\overline{w_1})|
   &\le \frac nm|\arg z_2-\arg w_2|
      +\frac1m|\arg v-\arg u|\\
   &<\frac{2(n+1)}{m}\omega<\theta_0.
\end{align*}
Thus both coordinate products lie in $\Gamma_{\theta_0}$. The same
argument restrictions also give
\[
   |\arg(v\overline u)|<2\omega<\theta_0,
\]
so $v\overline u\in\Gamma_{\theta_0}$. Moreover,
\[
   \frac{|z_1\overline{w_1}|^m}
        {|z_2\overline{w_2}|^n}
   =|v|\,|u|.
\]
If
\[
   \delta_0=\frac12\min\{r_*u_-,1-u_+\}>0,
\]
then
\[
   r_*u_-\le |v|\,|u|\le u_+,
\]
and therefore
\[
   \delta_0\le |v|\,|u|\le1-\delta_0.
\]
Lemma~\ref{lem:conical-lower}\textup{(i)} therefore gives
\[
   |B_{m/n}(z,w)|
   \gtrsim
   \frac{|z_2\overline{w_2}|^{-\kappa}}
        {|1-z_2\overline{w_2}|^2|1-v\overline u|^2}.
\]
Since $w_2$ remains in a fixed annulus,
$|z_2\overline{w_2}|^{-\kappa}\approx|z_2|^{-\kappa}$.  Moreover,
$|z_2\overline{w_2}|<1$ and $|v\overline u|<1$, so
\[
   |1-z_2\overline{w_2}|\le2,
   \qquad
   |1-v\overline u|\le2.
\]
Hence
\[
   |B_{m/n}(z,w)|^{c/2}
   \gtrsim |z_2|^{-c\kappa/2},
   \qquad w\in E.
\]
The positive continuous function $B_{m/n}(w,w)$ is bounded above
and below on the compact closure of $E$.  Therefore, for every fixed
$b\in\R$,
\[
   B_{m/n}(w,w)^{-b/2}\approx1,\qquad w\in E.
\]
Integration over $E$ gives \eqref{eq:fixed-test-lower}.
\end{proof}

\begin{lemma}\label{lem:pinching-boxes}
Fix
$0<\theta_0\le\theta_*$ and $0<\psi<\pi/2$.  There exist
$\eps_0>0$, constants $0<r<R<1$, and measurable sets
$E_\eps\subset H_{m/n}$, $0<\eps<\eps_0$, such that
\begin{equation}\label{eq:pinching-volume}
   |E_\eps|\approx\eps^2,
\end{equation}
and, for all $z,w\in E_\eps$,
\begin{equation}\label{eq:pinching-cones}
\begin{gathered}
   z_1\overline{w_1},\ z_2\overline{w_2}
   \in\Gamma_{\theta_0},\\
   1-z_2\overline{w_2},\; 
   1-\frac{(z_1\overline{w_1})^m}
           {(z_2\overline{w_2})^n}
   \in\Gamma_\psi.
\end{gathered}
\end{equation}
\begin{equation}\label{eq:pinching-moduli}
   r\le|z_2\overline{w_2}|\le R,
   \qquad
   R(z)\approx R(w)\approx\eps,
   \qquad
   \left|1-\frac{(z_1\overline{w_1})^m}
   {(z_2\overline{w_2})^n}\right|\approx\eps.
\end{equation}
All constants may depend on $m,n,\theta_0,\psi$, but not on
$\eps,z,w$.
\end{lemma}

\begin{proof}
Fix numbers $0<s_-<s_+<1$ and put $r=s_-^2$, $R=s_+^2$.  Choose
$\omega_2>0$ so small that
\[
   2\omega_2<\theta_0,
   \qquad
   \frac{2n}{m}\omega_2<\frac{\theta_0}{2},
   \qquad
   \frac{2\omega_2}{1-R}<\tan\psi.
\]
Let
\[
   U_2=\{\lambda=se^{i\vartheta}:
   s_-<s<s_+,\ |\vartheta|<\omega_2\}.
\]
Then $U_2\Subset\Dstar$.  Since $U_2$ is contained in a simply connected
subdomain of $\Dstar$, we fix on $U_2$ a single-valued branch
\[
   \tau(\lambda)=\lambda^{n/m}.
\]
For $z_2,w_2\in U_2$,
\[
   r\le|z_2\overline{w_2}|\le R,
   \qquad
   z_2\overline{w_2}\in\Gamma_{\theta_0},
   \qquad
   \tau(z_2)\overline{\tau(w_2)}
   \in\Gamma_{\theta_0/2}.
\]
Furthermore,
\[
   \operatorname{Re}(1-z_2\overline{w_2})\ge1-R,
   \qquad
   |\operatorname{Im}(1-z_2\overline{w_2})|<2\omega_2,
\]
and therefore
\[
   \frac{|\operatorname{Im}(1-z_2\overline{w_2})|}
        {\operatorname{Re}(1-z_2\overline{w_2})}
   <\tan\psi.
\]
Hence $1-z_2\overline{w_2}\in\Gamma_\psi$.

Set
\[
   c_\psi=\min\left\{1,\frac14\tan\psi\right\}.
\]
Choose $\eps_0>0$ so small that
\[
   \eps_0<\min\left\{\frac14,
   \frac{m\theta_0}{4c_\psi},\frac{\pi}{4c_\psi}\right\}.
\]
For $0<\eps<\eps_0$, define
\[
   Q_\eps=\{(1-s)e^{i\vartheta}:
   \eps<s<2\eps,\ |\vartheta|<c_\psi\eps\}.
\]
The principal branch of $u^{1/m}$ is single-valued on $Q_\eps$.  Define
\[
   E_\eps=
   \{(\tau(z_2)u^{1/m},z_2):z_2\in U_2,\ u\in Q_\eps\}.
\]
Then $z_1^m/z_2^n=u$, so $E_\eps\subset H_{m/n}$.

Let
\[
   z=(\tau(z_2)u^{1/m},z_2),
   \qquad
   w=(\tau(w_2)v^{1/m},w_2)
\]
be in $E_\eps$.  Since
\[
   |\arg(u^{1/m}\overline{v^{1/m}})|
   \le\frac{2c_\psi\eps}{m}<\frac{\theta_0}{2},
\]
the product representation
\[
   z_1\overline{w_1}
   =\tau(z_2)\overline{\tau(w_2)}
      u^{1/m}\overline{v^{1/m}}
\]
gives $z_1\overline{w_1}\in\Gamma_{\theta_0}$.  This proves the first two
cone inclusions in \eqref{eq:pinching-cones}.

Write
\[
   u=(1-s)e^{i\vartheta},
   \qquad
   v=(1-t)e^{i\varphi},
\]
where $\eps<s,t<2\eps$ and
$|\vartheta|,|\varphi|<c_\psi\eps$.  Then
\begin{align*}
   \operatorname{Re}(1-u\overline v)
   &=1-(1-s)(1-t)\cos(\vartheta-\varphi)\\
   &\ge s+t-st\ge\eps,
\end{align*}
and
\[
   |\operatorname{Im}(1-u\overline v)|
   \le|\vartheta-\varphi|<2c_\psi\eps.
\]
Because $2c_\psi<\tan\psi$, it follows that
$1-u\overline v\in\Gamma_\psi$.  Since
\[
   u\overline v
   =\frac{(z_1\overline{w_1})^m}
          {(z_2\overline{w_2})^n},
\]
all inclusions in \eqref{eq:pinching-cones} are proved.

The formula \eqref{eq:one-sheet-power-substitution} gives the exact
volume identity
\[
   |E_\eps|
   =\frac1{m^2}\int_{U_2}|z_2|^{2n/m}\dA(z_2)
      \int_{Q_\eps}|u|^{2/m-2}\dA(u).
\]
The factors involving $z_2$ and $|u|$ are bounded above and below uniformly,
and
\[
   |Q_\eps|
   =2c_\psi\eps\int_\eps^{2\eps}(1-s)\,ds
   \approx\eps^2.
\]
This proves \eqref{eq:pinching-volume}.

Finally,
\[
   1-|u|^2=2s-s^2\approx\eps,
   \qquad
   1-|v|^2=2t-t^2\approx\eps.
\]
Also,
\begin{align*}
   |1-u\overline v|
   &\le |1-(1-s)(1-t)|
      +(1-s)(1-t)|1-e^{i(\vartheta-\varphi)}|\\
   &\lesssim s+t+|\vartheta-\varphi|
   \lesssim\eps,
\end{align*}
whereas the real-part estimate above gives
$|1-u\overline v|\ge\eps$.  Hence
$|1-u\overline v|\approx\eps$, and
\eqref{eq:pinching-moduli} follows.
\end{proof}

\begin{proposition}\label{prop:positive-necessity}
If $S^+_{a,b,c}$ is bounded on $L^p(H_{m/n})$, then
$\mathcal C(a,b,c;p)$ holds.
\end{proposition}

\begin{proof}
We first derive the two left-hand inequalities in
\eqref{cond:boundary} and \eqref{cond:vertex}.  Choose fixed numbers
\[
   0<r_*<v_-<v_+<1,
   \qquad
   0<u_-<u_+<1,
   \qquad
   0<t_-<t_+<1,
\]
and choose $0<\theta_0\le\theta_*$.  Apply
Lemma~\ref{lem:fixed-test} with the tuple
\[
   (a,b,c;\theta_0,r_*,u_-,u_+,t_-,t_+),
\]
and denote the resulting objects by $E,J_2,J_\Phi$.  All fractional
powers below use the branches chosen in Lemma~\ref{lem:fixed-test}.

\smallskip
We next prove the inequality $-ap<1$.  Choose
$0<s_-<s_+<1$.  For $0<\eps<1-r_*$, let
\[
\begin{aligned}
   \mathcal U_\eps=\{&(z_2^{n/m}v^{1/m},z_2):
   s_-<|z_2|<s_+,\ \arg z_2\in J_2,\\
   &\hspace{36mm}
   1-\eps<|v|<1,\ \arg v\in J_\Phi\}.
\end{aligned}
\]
On $\mathcal U_\eps$, the $z_2$-variable stays in a compact annulus, while
$|v|>r_*$.  Lemma~\ref{lem:critical-diagonal} therefore gives
\[
   B_{m/n}(z,z)\approx
   \frac{|z_2|^{-2\kappa}}{R(z)^2\rho(z)^2}.
\]
Since $|z_2|$ and $\rho(z)$ are bounded above and below on $\mathcal U_\eps$,
it follows that
\[
   B_{m/n}(z,z)^{-a/2}\approx R(z)^a,
   \qquad R(z)=1-|v|^2,
   \qquad v=\Phi(z)=\frac{z_1^m}{z_2^n}.
\]
Together with Lemma~\ref{lem:fixed-test}, this yields
\begin{equation}\label{eq:positive-boundary-test}
   S^+_{a,b,c}\chi_E(z)\gtrsim R(z)^a,
   \qquad z\in\mathcal U_\eps.
\end{equation}
Since $\chi_E\in L^p$ and the operator is bounded, the right-hand side of
\eqref{eq:positive-boundary-test} must be locally in $L^p$.  By \eqref{eq:one-sheet-power-substitution}, the relevant local integral is,
up to a fixed positive factor,
\[
   \int_{1-\eps}^1(1-r^2)^{ap}r^{2/m-1}\,dr.
\]
Near $r=1$, the factor $r^{2/m-1}$ is bounded above and below.  The integral is
finite exactly when $ap>-1$.  Hence
\[
   -ap<1.
\]

\smallskip
We now prove the inequality
$(c-2a)(m+n-1)p<4(m+n)$. For $0<\eps<1/2$, set
\[
\begin{aligned}
   \mathcal V_\eps=\{&(z_2^{n/m}v^{1/m},z_2):
   0<|z_2|<\eps,\ \arg z_2\in J_2,\\
   &\hspace{36mm}
   v_-<|v|<v_+,\ \arg v\in J_\Phi\}.
\end{aligned}
\]
On $\mathcal V_\eps$, one has $R(z)\approx1$ and $\rho(z)\approx1$.
Lemma~\ref{lem:critical-diagonal} therefore gives
\[
   B_{m/n}(z,z)^{-a/2}\approx |z_2|^{a\kappa}.
\]
Lemma~\ref{lem:fixed-test} therefore implies
\begin{equation}\label{eq:positive-vertex-test}
   S^+_{a,b,c}\chi_E(z)
   \gtrsim |z_2|^{(a-c/2)\kappa},
   \qquad z\in\mathcal V_\eps.
\end{equation}
The ratio variable ranges over a fixed polar box of positive weighted area.
Using \eqref{eq:one-sheet-power-substitution} once more, local $L^p$
integrability of \eqref{eq:positive-vertex-test} is equivalent to
\[
   \int_0^\eps
   r^{(a-c/2)\kappa p+2n/m+1}\,dr<\infty.
\]
Thus
\[
   (a-c/2)\kappa p+\frac{2n}{m}+2>0,
\]
or, since $\kappa=(m+n-1)/m$,
\[
   (c-2a)(m+n-1)p<4(m+n).
\]

The adjoint of $S^+_{a,b,c}$ is $S^+_{b,a,c}$, because
$|B_{m/n}(z,w)|=|B_{m/n}(w,z)|$.  If $q=p/(p-1)$, boundedness on $L^p$
therefore implies boundedness of $S^+_{b,a,c}$ on $L^q$.  Applying the two
inequalities just proved with $(a,p)$ replaced by $(b,q)$ gives
\[
   -bq<1,
   \qquad
   (c-2b)(m+n-1)q<4(m+n).
\]
The first is equivalent to $1<(b+1)p$, and the second is equivalent to
\[
   4(m+n)
   <\bigl((2b-c)(m+n-1)+4(m+n)\bigr)p.
\]
This proves \eqref{cond:boundary} and \eqref{cond:vertex}.

\smallskip
It remains to prove the inequality $c\le a+b+2$.  Apply
Lemma~\ref{lem:pinching-boxes} with any fixed
$0<\psi<\pi/2$.  By \eqref{eq:conical-pinch-lower} and
\eqref{eq:pinching-moduli}, for $z,w\in E_\eps$,
\[
   |B_{m/n}(z,w)|
   \gtrsim
   \left|1-\frac{(z_1\overline{w_1})^m}
   {(z_2\overline{w_2})^n}\right|^{-2}
   \approx \eps^{-2}.
\]
Setting $w=z$ in \eqref{eq:pinching-moduli} gives
\[
   r\le |z_2|^2\le R,
   \qquad z\in E_\eps,
\]
and the same estimate holds with $z$ replaced by $w$.
Consequently,
\[
   |z_2|^{-2\kappa}\approx |w_2|^{-2\kappa}\approx1,
   \qquad
   \rho(z)\approx\rho(w)\approx1.
\]
The same formula gives $R(z)\approx R(w)\approx\eps$. Moreover,
$|\Phi(z)|,|\Phi(w)|\ge1/2$ when $\eps$ is sufficiently small. Therefore
Lemma~\ref{lem:critical-diagonal}, with $\delta=1/2$, gives
\[
   B_{m/n}(z,z)
   \approx \frac{|z_2|^{-2\kappa}}{R(z)^2\rho(z)^2}
   \approx \eps^{-2},
   \qquad
   B_{m/n}(w,w)
   \approx \frac{|w_2|^{-2\kappa}}{R(w)^2\rho(w)^2}
   \approx \eps^{-2}.
\]
Hence
\[
   B_{m/n}(z,z)^{-a/2}\approx\eps^a,
   \qquad
   B_{m/n}(w,w)^{-b/2}\approx\eps^b.
\]
Consequently,
\[
   S^+_{a,b,c}\chi_{E_\eps}(z)
   \gtrsim \eps^{a+b-c}|E_\eps|
   \approx\eps^{a+b+2-c},
   \qquad z\in E_\eps.
\]
Taking $L^p$ norms and cancelling the common factor
$|E_\eps|^{1/p}$ gives
\[
   \eps^{a+b+2-c}\lesssim1
   \qquad(\eps\to0^+).
\]
Therefore $c\le a+b+2$, proving \eqref{cond:pinch}.
\end{proof}

\begin{proof}[\textbf{Proof of Theorem~\ref{thm:positive-main}}]
Proposition~\ref{prop:sufficiency} proves sufficiency, and
Proposition~\ref{prop:positive-necessity} proves necessity.
\end{proof}

\begin{proof}[\textbf{Proof of Corollary~\ref{cor:berezin}}]
The Berezin transform is $S^+_{2,0,4}$.  For $(a,b,c)=(2,0,4)$,
condition \eqref{cond:boundary} is automatic once $p>1$, the left inequality
in \eqref{cond:vertex} is automatic because $c-2a=0$, and
\eqref{cond:pinch} holds with equality.  The remaining inequality is
\[
   4(m+n)<\bigl(-4(m+n-1)+4(m+n)\bigr)p=4p,
\]
which is equivalent to $p>m+n$.  Since $m+n\ge2$, this condition already
implies $p>1$.
\end{proof}

\section{Integer powers}

The sufficiency for $S_{a,b,2N}$ follows immediately from
Theorem~\ref{thm:positive-main}, since
\[
   |S_{a,b,2N}f(z)|\le S^+_{a,b,2N}(|f|)(z).
\]
It remains to prove necessity.

\begin{lemma}\label{lem:fixed-phase}
Let $N\in\Z_+$.  There exists $0<\theta_N\le\theta_*$ such that, whenever
\begin{equation}\label{eq:fixed-phase-region}
   z_1\overline{w_1},\quad z_2\overline{w_2},\quad
   1-z_2\overline{w_2},\quad
   1-\frac{(z_1\overline{w_1})^m}
           {(z_2\overline{w_2})^n}
   \in\Gamma_{\theta_N},
\end{equation}
one has
\begin{equation}\label{eq:fixed-phase-estimate}
   \operatorname{Re}\bigl(B_{m/n}(z,w)^N\bigr)
   \ge 2^{-1/2}|B_{m/n}(z,w)|^N.
\end{equation}
\end{lemma}

\begin{proof}
Set $\omega_N=\pi/(8N)$.  Choose $\theta_N>0$ so small that
\[
   \theta_N\le\min\{\theta_*,\theta(\omega_N)\}
   \qquad\text{and}\qquad
   \omega_N+2(n+2)\theta_N<\frac{\pi}{4N}.
\]
Under \eqref{eq:fixed-phase-region},
Lemma~\ref{lem:numerator-cone} places the common numerator
$\mathcal N(z,w)$ in $\Gamma_{\omega_N}$.  It is nonzero on this region, so
\[
   |\arg\mathcal N(z,w)|<\omega_N.
\]
The first denominator factor satisfies
\[
   \left|\arg\bigl((1-z_2\overline{w_2})^2\bigr)\right|<2\theta_N.
\]
For the second factor, write
\[
   (z_2\overline{w_2})^n-(z_1\overline{w_1})^m
   =(z_2\overline{w_2})^n
   \left(1-\frac{(z_1\overline{w_1})^m}
   {(z_2\overline{w_2})^n}\right).
\]
Its square has argument of absolute value less than
$2(n+1)\theta_N$.  Consequently,
\[
   |\arg B_{m/n}(z,w)|
   <\omega_N+2(n+2)\theta_N<\frac{\pi}{4N}.
\]
Thus
$\left|\arg\bigl(B_{m/n}(z,w)^N\bigr)\right|<\pi/4$, and
\eqref{eq:fixed-phase-estimate} follows from
$\cos(\pi/4)=2^{-1/2}$.
\end{proof}

\begin{lemma}\label{lem:one-minus-cone}
Let $0<\Theta<\pi/2$ and $0<R<1$.  If
$\lambda=re^{i\varphi}$, $r\le R$, and
\[
   |\varphi|<(1-R)\tan\Theta,
\]
then $1-\lambda\in\Gamma_\Theta$.
\end{lemma}

\begin{proof}
We have
\[
   \operatorname{Re}(1-\lambda)\ge1-R,
   \qquad
   |\operatorname{Im}(1-\lambda)|\le|\varphi|.
\]
Hence
\[
   \frac{|\operatorname{Im}(1-\lambda)|}
        {\operatorname{Re}(1-\lambda)}
   <\tan\Theta,
\]
and therefore $1-\lambda\in\Gamma_\Theta$.
\end{proof}

\begin{proposition}\label{prop:integer-necessity}
If $S_{a,b,2N}$ is bounded on $L^p(H_{m/n})$, then
$\mathcal C(a,b,2N;p)$ holds.
\end{proposition}

\begin{proof}
We shall use Lemma~\ref{lem:one-minus-cone} with
$\Theta=\theta_N$, where $\theta_N$ is supplied by
Lemma~\ref{lem:fixed-phase}.

Choose fixed numbers
\[
   0<r_*<v_-<v_+<1,
   \qquad
   0<u_-<u_+<1,
   \qquad
   0<t_-<t_+<1,
\]
and choose fixed $0<s_-<s_+<1$.  Put
\[
   R_2=s_+t_+<1,
   \qquad R_\Phi=u_+<1.
\]
Take $0<\theta_0\le\min\{\theta_*,\theta_N\}$ sufficiently small that
\[
   \theta_0<(1-R_2)\tan\theta_N,
   \qquad
   \theta_0<(1-R_\Phi)\tan\theta_N.
\]
Apply Lemma~\ref{lem:fixed-test} with
\[
   (a,b,2N;\theta_0,r_*,u_-,u_+,t_-,t_+).
\]
We obtain
$E\Subset H_{m/n}$ and intervals $J_2,J_\Phi$ such that, for every
$w=(w_2^{n/m}u^{1/m},w_2)\in E$ and every
\[
   z=(z_2^{n/m}v^{1/m},z_2)
\]
with $\arg z_2\in J_2$, $\arg v\in J_\Phi$, and $r_*<|v|<1$,
Lemma~\ref{lem:fixed-test} gives
\[
   z_1\overline{w_1},\quad z_2\overline{w_2},\quad
   v\overline u\in\Gamma_{\theta_0}\subset\Gamma_{\theta_N}.
\]

Consider first the sets $\mathcal U_\eps$ from the proof of
Proposition~\ref{prop:positive-necessity}, using the fixed annulus
$s_-<|z_2|<s_+$.  For $z\in\mathcal U_\eps$ and $w\in E$,
\[
   |z_2\overline{w_2}|\le R_2,
   \qquad
   |v\overline u|\le R_\Phi.
\]
Lemma~\ref{lem:one-minus-cone} gives
\[
   1-z_2\overline{w_2}\in\Gamma_{\theta_N},
   \qquad
   1-v\overline u\in\Gamma_{\theta_N}.
\]
Since
\[
   v\overline u
   =\frac{(z_1\overline{w_1})^m}
          {(z_2\overline{w_2})^n},
\]
all four conditions in \eqref{eq:fixed-phase-region} hold.  The diagonal
weights are positive real numbers, so Lemma~\ref{lem:fixed-phase} yields
\begin{equation}\label{eq:integer-positive-comparison-fixed}
\begin{aligned}
   |S_{a,b,2N}\chi_E(z)|
   &\ge \operatorname{Re}S_{a,b,2N}\chi_E(z)\\
   &\ge 2^{-1/2}S^+_{a,b,2N}\chi_E(z).
\end{aligned}
\end{equation}
For all sufficiently small $\eps$, the same argument applies on the sets
$\mathcal V_\eps$, because their second coordinates are smaller and their ratio
variables remain in $v_-<|v|<v_+$.  Therefore the two calculations
leading from
\eqref{eq:positive-boundary-test} and \eqref{eq:positive-vertex-test} to the
left-hand inequalities in \eqref{cond:boundary} and \eqref{cond:vertex} apply
verbatim with $c=2N$.  We obtain
\[
   -ap<1,
   \qquad
   (2N-2a)(m+n-1)p<4(m+n).
\]

To prove the inequality $2N\le a+b+2$, apply Lemma~\ref{lem:pinching-boxes} with
\[
   0<\theta_0\le\min\{\theta_*,\theta_N\},
   \qquad \psi=\theta_N.
\]
By \eqref{eq:pinching-cones}, all four quantities in
\eqref{eq:fixed-phase-region} belong to $\Gamma_{\theta_N}$ for every
$z,w\in E_\eps$.  Hence Lemma~\ref{lem:fixed-phase}, together with the kernel, diagonal, and
volume estimates established in the proof of
Proposition~\ref{prop:positive-necessity}, gives
\[
   |S_{a,b,2N}\chi_{E_\eps}(z)|
   \ge \operatorname{Re}S_{a,b,2N}\chi_{E_\eps}(z)
   \gtrsim \eps^{a+b-2N}|E_\eps|
   \approx\eps^{a+b+2-2N}
\]
for $z\in E_\eps$.  Boundedness and the norm comparison with
$\chi_{E_\eps}$ imply
\[
   2N\le a+b+2.
\]

Finally, the adjoint of $S_{a,b,2N}$ is $S_{b,a,2N}$, by the Hermitian
symmetry of the Bergman kernel.  Applying the two left-hand inequalities to the
adjoint on $L^q$, $q=p/(p-1)$, gives
\[
   1<(b+1)p
\]
and
\[
   4(m+n)
   <\bigl((2b-2N)(m+n-1)+4(m+n)\bigr)p.
\]
Together with the preceding inequalities, these are precisely
$\mathcal C(a,b,2N;p)$.
\end{proof}

\begin{proof}[\textbf{Proof of Theorem~\ref{thm:integer-main}}]
The equivalence of \textup{(i)} and \textup{(iii)} is
Theorem~\ref{thm:positive-main} with $c=2N$.  The implication
\textup{(i)}$\Rightarrow$\textup{(ii)} follows from the pointwise domination
\[
   |S_{a,b,2N}f|\le S^+_{a,b,2N}(|f|).
\]
The implication \textup{(ii)}$\Rightarrow$\textup{(iii)} is
Proposition~\ref{prop:integer-necessity}.
\end{proof}

\begin{proof}[\textbf{Proof of Corollary~\ref{cor:bergman-projection}}]
The Bergman projection is $S_{0,0,2}$, so we apply
Theorem~\ref{thm:integer-main} with $(a,b,N)=(0,0,1)$.  Condition
\eqref{cond:boundary} is automatic for $1<p<\infty$, and
\eqref{cond:pinch} holds with equality.  The two inequalities in
\eqref{cond:vertex} become
\[
   2(m+n-1)p<4(m+n)
\]
and
\[
   4(m+n)<\bigl(4(m+n)-2(m+n-1)\bigr)p
            =2(m+n+1)p.
\]
They are equivalent, respectively, to
\[
   p<\frac{2m+2n}{m+n-1}
   \qquad\text{and}\qquad
   p>\frac{2m+2n}{m+n+1}.
\]
\end{proof}

\end{document}